\documentclass{amsart}
\usepackage{amsmath, amsfonts, amssymb, setspace}
\def\a{\alpha}

\def\G{\Gamma}

\def\k{\kappa}

\newcommand{\mB}{\mathcal{B}}

\newcommand{\mS}{\mathcal{S}}

\newcommand{\fa}{\mathfrak{a}}

\newcommand{\fM}{\mathfrak{M}}

\newcommand{\fN}{\mathfrak{N}}

\newcommand{\fp}{\mathfrak{p}}

\newcommand{\fq}{\mathfrak{q}}

\newcommand{\bfC}{\mathbf{C}}

\newcommand{\bfF}{\mathbf{F}}

\newcommand{\bfH}{\mathbf{H}}

\newcommand{\bfQ}{\mathbf{Q}}

\newcommand{\bfT}{\mathbf{T}}

\newcommand{\bfZ}{\mathbf{Z}}

\newcommand{\Oo}{\mathcal{O}}

\newcommand{\OF}{\mathcal{O}_F}

\newcommand{\OK}{\mathcal{O}_K}

\newcommand{\AF}{\mathbf{A}_F}

\newcommand{\AFf}{\mathbf{A}_{F,\textup{f}}}

\newcommand{\OFq}{\mathcal{O}_{F,\mathfrak{q}}}

\newcommand{\ov}{\overline}

\newcommand{\be}{\begin{equation}}
\newcommand{\ee}{\end{equation}}
\newcommand{\bes}{\begin{equation*}}
\newcommand{\ees}{\end{equation*}}

\newcommand{\bs}{\begin{split}}
\newcommand{\es}{\end{split}}
\newcommand{\bss}{\begin{split*}}
\newcommand{\ess}{\end{split*}}

\newcommand{\bmat}{\left[ \begin{matrix}}
\newcommand{\emat}{\end{matrix} \right]}
\newcommand{\bsmat}{\left[ \begin{smallmatrix}}
\newcommand{\esmat}{\end{smallmatrix} \right]}

\newcommand{\bml}{\begin{multline}}
\newcommand{\eml}{\end{multline}}
\newcommand{\bmls}{\begin{multline*}}
\newcommand{\emls}{\end{multline*}}

\DeclareMathOperator{\Cl}{Cl}

\DeclareMathOperator{\End}{End}

\DeclareMathOperator{\GL}{GL}

\DeclareMathOperator{\Hom}{Hom}

\DeclareMathOperator{\image}{Im}
\DeclareMathOperator{\Ind}{Ind}

\DeclareMathOperator{\SL}{SL}

\DeclareMathOperator{\Sym}{Sym}

\def\c{\chi}

\def\quf{\mathbf{Q}}

\def\ideleF{\mathbf{A}_F^{\times}}

\def\adeleF{\mathbf{A}_F}
\def\adeleFf{\mathbf{A}_{F,\textup{f}}}

\def\bbf{\mathbf}

\newcommand{\hs}{\hspace{2pt}}
\newcommand{\hf}{\hspace{5pt}}

\newcommand{\Op}{\mathcal{O}_{(\mathfrak{p})}}

\newcommand{\iy}{\infty}

\usepackage[all]{xy}
\usepackage{amsthm}
\usepackage{amssymb, amsfonts}
\SelectTips{cm}{10}\UseTips
\theoremstyle{plain}
\newtheorem{thm}{Theorem}

\newtheorem{cor}[thm]{Corollary}
\newtheorem{lemma}[thm]{Lemma}

\theoremstyle{definition}

\newtheorem{rem}[thm]{Remark}

\numberwithin{thm}{section}
\numberwithin{equation}{section}

\begin{document}
\author{Krzysztof Klosin}
\title{On Ihara's lemma for degree one and two cohomology over imaginary quadratic fields}
\footnotetext{2010 Mathematics subject classification 11F33 (primary), 11F75 (secondary).}
\date{February 12, 2013}

\begin{abstract}
We prove a version of Ihara's Lemma for degree $q=1,2$ cuspidal cohomology of the symmetric space attached to automorphic forms of arbitrary weight ($k\geq 2$) over an imaginary quadratic field with torsion (prime power) coefficients. This extends an earlier result of the author \cite{Klosin08} which concerned the case $k=2$, $q=1$. Our method is different from \cite{Klosin08} and uses results of Diamond \cite{Diamond91} and Blasius-Franke-Grunewald \cite{BlasiusFrankeGrunewald94}.  We discuss the relationship of our main theorem to the problem of the existence of level-raising congruences.
%A crucial input in the degree two case is a result of Blasius, Franke and Grunewald on the vanishing of the restriction map from the cohomology of $S$-arithmetic groups to the cohomology of congruence subgroups of $\SL_2(F)$ \cite{BlasiusFrankeGrunewald94}. 
\end{abstract}

\maketitle

\section{Introduction}
The classical Ihara's lemma states that the kernel of the map $\alpha: J_0(N)^2 \to J_0(Np)$ is Eisenstein if $(N,p)=1$, where $J_0(N')$ denotes the Jacobian of the compatification of the modular curve $\Gamma_0(N') \setminus \bfH$ and $\alpha$ is the sum of the two standard $p$-degeneracy maps. In \cite{Klosin08} the author proved an analogue of this result for degree one parabolic cohomology arising from weight 2 automorphic forms over an imaginary quadratic field $F$. More precisely, for $n=0,1$ let $Y_n$  be the analogue over $F$ of the modular curve $Y_0(Np^n)$ (for precise definitions cf. section \ref{Preliminaries}). Write  $H^q_!(Y_n, \tilde{M}_n)$ for the degree $q$ parabolic cohomology group, where $\tilde{M}_n$ are 
sheaves of sections of the topological covering $\GL_2(F) \setminus 
(V\times M) \rightarrow Y_n$ with $M$ a torsion 
$\mathbf{Z}[\GL_2(F)]$-module 
and $V=\GL_2(\adeleF)/K_n \cdot U_2(\mathbf{C}) \cdot \mathbf{C}^{\times}$ 
for $K_n$ a 
compact 
subgroup, which is an analogue of $\G_0(Np^n)$. 
For a prime ideal $\fp$ of the ring of integers 
$\OF$ of $F$ we have two standard $\fp$-degeneracy maps 
whose sum $H^q_!(Y_0, \tilde{M}_0)^2 \rightarrow H^q_!(Y_1, \tilde{M}_1)$ 
we will call $\alpha_q$.
The main result of \cite{Klosin08} (Theorem 2) then asserts that the 
kernel of $\alpha_1$ is Eisenstein when $M$ is a trivial $\GL_2$-module (weight 2 case) and that in this case $\alpha_q$ is injective  when $M$ has exponent a power of $p$ (cf. [loc.cit.] Remark 12). 
%\com{Remove:}
%In the current paper we prove the following the following two results:
%\com{added:} 

In the current paper we prove an analogous result for the maps $\alpha_q$ for $q=1,2$ for all weights $k \geq 2$. Let us now state the main results in a more precise form. For a positive integer $m$ and an $\OF$-module $M$ let $L(m,M)$ be the module of homogeneous polynomials of degree $m$ in two variables over $F$. Let $\fN$ be an ideal of $\OF$, $\fp\nmid \fN$ a prime ideal of $\OF$ and write $Y_n$ for the symmetric space of level $\fN \fp^n$ (for precise definitions see section \ref{Preliminaries}). Let $\ell\nmid \#\OF^{\times}$ be a prime and let $E$ be a (sufficiently large) finite extension of $\bfQ_{\ell}$ with valuation ring $\Oo$. Then the following Theorem is the main result of this paper (Ihara's lemma for $H^1$ and $H^2$).
\begin{thm} [Theorems \ref{IharaH2} and \ref{IharaH1}] Let $q=1$ or $q=2$. If $q=1$ suppose that $\ell \nmid N(\fN)\#(\OK/\fN \OK)^{\times}$ and that $\ell >k-2$. Then the  map $$\alpha_q: H^q_! (Y_0, \tilde{L}(k-2, E/\Oo))^2 \to H^q_!(Y_1, \tilde{L}(k-2, E/\Oo))$$ is injective. \end{thm}
%\begin{thm} [Theorem \ref{IharaH2}] With $\ell$ as above the standard $\fp$-degeneracy map $$\alpha_2: H^2_! (Y_0, \tilde{L}(k-2, E/\Oo)^2 \to H^2_!(\tilde{L}(k-2, E/\Oo))$$ is injective. \end{thm}
%\com{Remove:}
%\begin{enumerate} \item We show that $\alpha_1$ is injective for all weights $k \geq 2$ if $M$ has exponent a power of $p$. The method applied in  \cite{Klosin08} used modular symbols and relied on the assumption that the group action on $M$ was trivial. To treat the case of a general weight we use an approach developed by Diamond \cite{Diamond91} for the $\bfQ$-case and we extend it here to the case of $F$ using some input from \cite{SengunTurkelli09}. This is carried out in section \ref{sectionI2}. \item We show that $\alpha_2$ is injective for all weights $k\geq 2$ if $M$ has exponent a power of $p$ using results of Blasius, Franke and Grunewald \cite{BlasiusFrankeGrunewald94} on the vanishing of the restriction map from group cohomology of $S$-arithmetic groups to group cohomology of arithmetic groups. This is carried out in section \ref{sectionI1}. \end{enumerate}
%\com{End remove}
%\com{Added:} 
The method applied in  \cite{Klosin08} for $q=1$ used modular symbols and relied on the assumption that in this case the sheaf $\tilde{L}(k-2, E/\Oo)$ is constant. To treat the case of a general weight we use an approach developed by Diamond \cite{Diamond91} for the $\bfQ$-case and we extend it here to the case of $F$ using some input from \cite{SengunTurkelli09}. This is carried out in section \ref{sectionI2}. For $q=2$ we use results of Blasius, Franke and Grunewald \cite{BlasiusFrankeGrunewald94} on the vanishing of the restriction map from group cohomology of $S$-arithmetic groups to group cohomology of arithmetic groups. This is carried out in section \ref{sectionI1}.

%\com{end added}
In \cite{Ribet84} and \cite{Diamond91} Ihara's Lemma was used to show the existence of level-raising congruences, i.e., congruences  between modular forms of level $N$ and those of level $Np$. Our results however cannot be used to prove such a result. The problem is the occurrence of torsion in the degree two cohomology which prevents a  `lifting' of our result for cohomology with torsion coefficients to a statement about lattices generated by eigenforms in the spaces of automorphic forms. It is also not possible to conclude a level-raising result on the level of cohomology itself because the cup product pairing is only perfect modulo torsion hence preventing us from using the standard technique of composing $\alpha_q$ with its adjoint $\alpha_q^+$ and relating the order of $\ker(\alpha_q^+\alpha_q)$ to the order of the congruence module. The relationship of our results to the problem of level-raising is discussed in detail in section \ref{The module of congruences}.

On the other hand a level-raising result for torsion cohomology classes with trivial coefficients %\com{added:} 
%(see Theorem \ref{CalVen}) 
%\com{end:added}
 has recently been obtained by Calegari and Venkatesh \cite{CalegariVenkateshbook}. As noted in [loc.cit.] the cohomology classes of `raised' level constructed in [loc.cit.] do not always lift to characteristic zero hence in fact a level-raising result of the type proved in \cite{Ribet84} and \cite{Diamond91} is not to be expected in the context of automorphic forms over imaginary quadratic fields. We would like thank Frank Calegari and Akshay Venkatesh for sending us an early version of their book.

\section{Preliminaries} \label{Preliminaries}
\subsection{The congruence subgroups of $\GL_2(F)$ and symmetric spaces} Let $F$ be an imaginary quadratic extension 
of $\quf$ and denote by $\OF$ its ring of integers. Fix once and for all an embedding $\ov{F} \hookrightarrow \bfC$. For any ideal $\fM\subset \OF$ we will write $\Phi(\fM)$ for the integer $N(\fM) \cdot \#(\OF/\fM)^{\times}$, where $N$ denotes the absolute norm. Let $\fN$ be an ideal 
of 
$\OF$ such that the $\mathbf{Z}$-ideal $\fN \cap \mathbf{Z}$ has a 
generator greater than 3. Let $\fp$ be a prime ideal such that $\fp 
\nmid\fN$. Write $p$ for its residue characteristic.
Denote 
by $\Cl_F$ the class group of $F$ and choose representatives of distinct 
ideal classes to be prime ideals $\fp_j$, 
$j=1, \dots, \#\Cl_F$, relatively prime to both $\fN$ and $\fp$. Let 
$\tilde\pi$, 
(resp. $\tilde\pi_j$) be a uniformizer of the completion $F_{\fp}$ (resp. 
$F_{\fp_j}$) of $F$ at the prime $\fp$ (resp. $\fp_j$), and put 
$\pi$ (resp. $\pi_j$) to be the idele $(\dots, 1, \tilde\pi, 1, 
\dots) \in \ideleF$ (resp. $(\dots, 1, \tilde\pi_j, 1,
\dots) \in \ideleF$), where $\tilde\pi$ (resp. $\tilde\pi_j$) occurs at the $\fp$-th 
place 
(resp. $\fp_j$-th place). We will write $\eta$ for $\bmat \pi \\ & 1\emat \in \GL_2(\AFf)$.
%\com{Needs to be in this way - agrees with Taylor, but not sure if it agrees with Khare}

For each $n\in \mathbf{Z}_{\geq 0}$, we define compact open 
subgroups of $\GL_2(\adeleFf)$
$$K_n:= \left\{\bmat a&b \\ c&d \emat \in \prod_{\fq \nmid \iy} 
\GL_2(\OFq) \mid c \in \fN \fp^n \right\}.$$
Here $\adeleFf$ denotes the finite adeles of $F$ and $\OFq$ the 
ring of integers of $F_{\fq}$. For $n \geq 0$ we 
also set $K_n^{\fp} 
= \eta^{-1} K_n \eta$.

For any compact open 
subgroup $K$ of $\GL_2(\adeleFf)$ we put $Y_K = \GL_2(F) \setminus 
\GL_2(\adeleF)/ K \cdot U_2(\mathbf{C}) \cdot Z_{\iy}$, 
where $Z_{\iy} = \mathbf{C}^{\times}$ is the center of 
$\GL_2(\mathbf{C})$ and $U(2):= \{M \in \GL_2(\mathbf{C}) \mid M \ov{M}^t 
= 
I_2\}$ (here `bar' denotes complex conjugation and $I_2$ stands for the $2 
\times 
2$-identity matrix). If $K$ is sufficiently large (which will be the 
case for all compact open subgroups we will consider) this space is a 
disjoint union of $\#\Cl_F$ connected 
components $Y_K = \coprod_{j=1}^{\# \Cl_F} (\Gamma_K)_j \setminus 
\mathcal{Z}$, where $\mathcal{Z} = \GL_2(\mathbf{C})/U_2(\mathbf{C}) 
\mathbf{C}^{\times}$ and $(\Gamma_K)_j = \GL_2(F) \cap \bmat \pi_j\\ 
&1 
\emat 
K \bmat \pi_j\\ &1 \emat^{-1}$. To ease notation we put $Y_n := 
Y_{K_n}$, $Y_n^{\fp} := Y_{K_n^{\fp}}$, $\G_{n,j}:= (\G_{K_n})_j$ and 
$\G^{\fp}_{n,j}:= (\G_{K_n^{\fp}})_j$.

We have the following diagram:
\be \label{diagramK} \xy*!C\xybox{\xymatrix{\dots \ar[r]^{\subset} 
&K_{n+1}\ar[r]^{\subset}\ar[dr]^{\subset}&
K_n \ar[r]^{\subset} \ar[dr]^{\subset} & K_{n-1}\ar[r]^{\subset} & 
\dots\\
\dots \ar[r]^{\subset} &K^{\fp}_{n+1}\ar[r]^{\subset}\ar[u]^{\wr}&
K^{\fp}_n \ar[r]^{\subset} \ar[u]^{\wr} & 
K^{\fp}_{n-1}\ar[r]^{\subset} \ar[u]^{\wr}&
\dots
  }}\endxy \ee
where the horizontal and diagonal arrows are inclusions and the 
vertical arrows are conjugation by $\eta$. Diagram (\ref{diagramK}) is not commutative, but it is 
``vertically commutative'', by which we mean that given two objects in the 
diagram, two directed paths between those two objects define the same map 
if and only if the two paths contain the same number of vertical arrows.  

Diagram (\ref{diagramK}) induces the following vertically commutative 
diagram of the corresponding symmetric spaces:
\be \label{diagramX}
\xy*!C\xybox{\xymatrix{\dots \ar[r] 
&Y_{n+1}\ar[r]\ar[dr]&
Y_n \ar[r] \ar[dr] & Y_{n-1}\ar[r] & 
\dots\\
\dots \ar[r] &Y^{\fp}_{n+1}\ar[r]\ar[u]^{\wr}&
Y^{\fp}_n \ar[r] \ar[u]^{\wr} & 
Y^{\fp}_{n-1}\ar[r] \ar[u]^{\wr}&
\dots
  }}\endxy \ee

The horizontal and diagonal arrows in diagram (\ref{diagramX}) are the 
natural 
projections and the vertical arrows are maps given by $(g_{\iy}, g_f) 
\mapsto \left(g_{\iy}, g_f \eta \right)$.

\subsection{Cohomology} 
Let $M$ be a finitely generated
 $\bfZ$-module with $M^{\rm tor}$ of exponent relatively prime to 
$\#\OF^{\times}$ endowed with a $\GL_2(F)$-action. Denote by $\tilde{M}_K$ 
the sheaf of continuous sections of the 
topological covering $\GL_2(F)\setminus [(\GL_2(\adeleF)/K 
\cdot U_2(\mathbf{C}) \cdot 
Z_{\iy})\times M] \rightarrow Y_K$, where $\GL_2(F)$ acts diagonally on 
$(\GL_2(\adeleF)/K U_2(\mathbf{C})
Z_{\iy})\times M$. Here $M$ is equipped with the discrete topology. 
As above, we put $\tilde{M}_n:= \tilde{M}_{K_n}$ and 
$\tilde{M}^{\fp}_n:=\tilde{M}_{K^{\fp}_n}$.

Given a surjective map $\phi: Y_K \rightarrow Y_{K'}$, we get an 
isomorphism of 
sheaves 
$\phi^{-1}\tilde{M}_{K'} \xrightarrow{\sim} \tilde{M}_K$, which yields a 
map on 
cohomology $$H^q(Y_{K'}, \tilde{M}_{K'}) \rightarrow H^q(Y_{K}, 
\phi^{-1}\tilde{M}_{K'}) \cong H^q(Y_K, \tilde{M}_{K}).$$ Hence diagram 
(\ref{diagramX}) gives rise to a vertically commutative diagram of 
cohomology groups:
\be \label{diagramc} \xymatrix@C4em@R5em{\dots 
&H^q(Y_{n+1}, \tilde{M}_{n+1})\ar[l]\ar[d]_{\wr}^{\a_{\fp}^{n+1}}&
H^q(Y_n,\tilde{M}_n) \ar[l]_{\a_1^{n,n+1}} \ar[d]_{\wr}^{\a_{\fp}^{n}}& 
H^q(Y_{n-1}, \tilde{M}_{n-1})\ar[l]_{\a_1^{n-1,n}}\ar[d]_{\wr}^{\a_{\fp}^{n-1}}&
\dots\ar[l]\\
\dots &H^q(Y^{\fp}_{n+1}, 
\tilde{M}^{\fp}_{n+1})\ar[l]&
H^q(Y^{\fp}_n, \tilde{M}^{\fp}_n) \ar[l]_{\a_1^{n\fp, (n+1)\fp}} \ar[ul]_{\a_1^{n, (n+1)\fp}}
 &
H^q(Y^{\fp}_{n-1}, \tilde{M}^{\fp}_{n-1})\ar[l]_{\a_1^{(n-1)\fp, n\fp}} \ar[ul]_{\a_1^{n-1, 
n\fp}}
&
\dots\ar[l]
  } \ee

These sheaf cohomology groups can be related to the group cohomology 
of $\G_{n,j}$ and $\G_{n,j}^{\fp}$ with coefficients in $M$. In fact, for 
each compact open 
subgroup $K$ with corresponding decomposition $Y_K = \coprod_{j=1}^{\# 
\Cl_F} 
(\Gamma_K)_j \setminus
\mathcal{Z}$, we have the following 
commutative diagram in which the horizontal maps are inclusions:
\be\label{diagramgc}\xy*!C\xybox{\xymatrix{H^q_!(Y_K, \tilde{M}_K) \ar[r]  
&H^q(Y_K, 
\tilde{M}_K) \\
\bigoplus_{j=1}^{\#\Cl_F} H^q_P((\G_K)_j, M) \ar[r] \ar[u]& 
\bigoplus_{j=1}^{\#\Cl_F} H^q((\G_K)_j, M) \ar[u]}}\endxy\ee 
Here $H^q_!(Y_K, \tilde{M}_K)$ denotes the image of the cohomology 
with compact support $H^q_c(Y_K, \tilde{M}_K)$ inside $H^q(Y_K, 
\tilde{M}_K)$ and $H^q_P$ denotes the parabolic cohomology, i.e., 
$H^q_P((\G_K)_j, M):= \ker(H^q((\G_K)_j, M)\rightarrow 
\bigoplus_{B \in \mB_j} H^q((\G_K)_{j,B}, M))$, where $\mB_j$ is a 
fixed set of representatives of $(\G_K)_j$-conjugacy classes of 
Borel subgroups of 
$GL_2(F)$ and $(\G_K)_{j,B}:= (\G_K)_{j} \cap B$. The vertical 
arrows in diagram (\ref{diagramgc}) are isomorphisms provided that there 
exists a torsion-free normal 
subgroup of $(\G_K)_j$ of finite index relatively prime to 
the exponent of $M^{\rm tor}$. If $K=K_n$ or $K=K_n^{\fp}$, $n\geq 0$, this 
condition is satisfied because of our assumption that $\fN \cap 
\mathbf{Z}$ has a generator 
greater than 3 and the exponent of $M^{\rm tor}$ is relatively prime to $\# 
\OF^{\times}$ (cf. \cite{Urban95}, 
section 2.3). In what follows 
we may therefore identify the sheaf cohomology with the group cohomology. 
Note that all maps in 
diagram (\ref{diagramc}) preserve parabolic cohomology. The maps 
$\a_1^{*,*}$ are the natural restriction maps on group cohomology, so 
in 
particular they preserve the decomposition $\bigoplus_{i=1}^{\#\Cl_F} 
H^q((\G_K)_j, M)$.  
%\com{note our $Y$ is Urban's $X$ - non-compactified space - change all $Y$'s to $X$'s}

Let us introduce one 
more 
group:
$$K_{-1}:= \left\{\bmat a&b \\ c&d \emat \in \GL_2(F_{\fp})\times 
\prod_{\fq
\nmid \fp \iy}
\GL_2(\OFq) \mid c \in \fN , \hs ad-bc \in \prod_{\fq \nmid \iy}
\OFq^{\times}
\right\}.$$
The group $K_{-1}$ is not compact, but we can still define 
$$\G_{-1,j}:= 
\GL_2(F) \cap \bmat \tilde{\pi}_j \\ &1 \emat K_{-1} \bmat \tilde{\pi}_j
\\ 
&1 \emat^{-1}$$
for $j=1, \dots \#\Cl_F$. Note that $\Gamma_{-1,j}$ are not discrete subgroups of $\SL_2(\bfC)$. They are commensurable with an $S$-arithmetic subgroup (in the sense of \cite{Serre70}) of $\SL_2(F)$ where $S=\{\fp\}$.  However it still makes sense to define the group  cohomology groups $H^q(\Gamma_{-1,j},M)$ as well as the subgroups of parabolic cohomology $H^q_P(\Gamma_{-1,j}, M)$. 

The sheaf and group cohomologies are in a natural way modules over the 
corresponding Hecke 
algebras. (For the definition of the Hecke action on cohomology, see 
\cite{Urban95} or \cite{Hida93}). Here we will only consider the 
subalgebra 
$\mathbf{T}_{n,\bfZ}$ of the full Hecke algebra which is generated over $\bbf{Z}$ 
by the 
double cosets $T_{\fp'}:=K \bmat\pi'\\
&1\emat K$ %and $T_{\fp',\fp'}:=K \bmat \pi' \\ &\pi' \emat K$ 
for $\pi'$ 
a uniformizer of $F_{\fp'}$ with $\fp'$ running over prime ideals of $\OF$ 
such that 
$\fp' \nmid \fN \fp$. For a $\bfZ$-algebra $A$ we set $\bfT_{n,A}:= \bfT_{n, \bfZ} \otimes_{\bfZ} A$.

\subsection{The sheaf and cup product pairing} \label{The Eichler 1}

Let $k\geq 2$.
%\com{Remove:} Let $\ell>k-1$ be a prime not dividing $N(\fN\fp)D_F\#\OF^{\times}$, where $D_F$ is the discriminant of $F$ \com{End remove}
%\com{Added:}
Let $\ell$ be a prime.
%\com{end added}
%Moreover, we also assume that $\#\OF^{\times}\mid 2(k-2)$ - note that in particular this condition is satisfied for arbitrary $k \geq 2$ if $F=\bfQ(\sqrt{m})$ with square-free $m\neq -1,-3$. 
%Fix a prime $\fl$ of $F$ over $\ell$. Write $F_{\fl}$ for the $\fl$-adic completion of $F$. 
Fix an isomorphism $\ov{\bfQ}_{\ell} \cong \bfC$. Let $E\subset \ov{\bfQ}_{\ell}$ be an (always sufficiently large) finite extension of 
%$F_{\fl}$
$\bfQ_{\ell}$. In particular we assume that $E$ contains $F$ and all the eigenvalues of all Hecke eigenforms for a fixed $k$ and level (this is possible since the extension of $\bfQ$ generated by these eigenvalues is a number field - cf. e.g., Theorem A in \cite{Taylor94}).
Write $\Oo$ for the valuation ring of $E$, $\varpi$ for a choice of a uniformizer and $\bfF$ for the residue field. %We fix an isomorphism $\ov{F}_{\fl} \cong \bfC$, so we can regard $E$ as a subfield of $\bfC$. 

Let $A$ be an $\OF$-algebra. For an integer $m\geq 0$, write $\Sym^m (A)$ for the ring of homogeneous polynomials in two variables of degree $m$ with coefficients in $A$. For a  subgroup $\Gamma \subset \GL_2(F)$, $\gamma =\bmat a&b \\ c&d\emat \in \Gamma$ and $P(X,Y) \in \Sym^m(A)$ we set $(\gamma P)(X,Y) := P(aX+cY, bX+dY)$. If $P(aX+cY, bX+dY) \in A[X,Y]$ for every $\gamma \in \Gamma$, this defines an action of $\Gamma$ on $\Sym^m (A)$. Set $$L(m,A) =  \Sym^m(A) \otimes_{\OF} \Sym^m(A),$$ where the second factor is an $\OF$-module via the non-trivial automorphism $a \mapsto \ov{a}$ of $F$. Set $\gamma (P\otimes Q) := \gamma P \otimes \ov{\gamma} Q$. Put $$L(m,M):= L(m, A) \otimes_{A} M$$ for any $A$-module $M$. If $M$ is an $A$-algebra, then $L(m,M)$ is a ring.

%\com{Remove: Now assume that $A$ is an $\OF$-algebra in which all rational primes $q \leq m$ are invertible (cf. \cite{Ghate02}, section 3.3).}
%\com{Added:}
 Now assume that $\ell>k-1$ is a prime not dividing $N(\fN\fp)D_F\#\OF^{\times}$, where $D_F$ is the discriminant of $F$. Moreover, assume that $A$ is an $\OF$-algebra in which all rational primes $q \leq m$ are invertible (cf. \cite{Ghate02}, section 3.3).
%\com{End added}
 For a sheaf $\tilde{M}$ of $A$-modules we set $\tilde{M}^0 = \Hom(\tilde{M}, A)$ (cf. \cite{Urban98}, p.288). Ubran ([loc.cit], p.299) shows that it is possible to view $\tilde{L}(m,A)^0$ as a subsheaf of $\tilde{L}(m,A)$. Indeed, if $\Sym^m (A)$ is endowed with an action of a congruence subgroup $\Gamma \subset \GL_2(F)$, one defines (cf. e.g. \cite{Ghate02}, section 3.3) a natural pairing $[\cdot , \cdot]_{m}$ on $\Sym^m (A) \otimes \Sym^m (A) \to A$ by $$\left[\sum_{i=0}^{m} a_i X^i Y^{m-i}, \sum_{i=0}^{m} b_i X^i Y^{m-i}\right]_m:= \sum_{i=0}^m (-1)^i \frac{a_i b_{m-i}}{{m \choose i}}.$$ Then we can define a $\Gamma$-equivariant pairing on $L(m,A) \otimes L(m, A) \to A$ by $[P\otimes P', Q\otimes Q']_m:= [P,Q]_m [P',Q']_m$. %This pairing is $\Gamma$-equivariant if $(\det \Gamma)^m=\{1\}$. If $m=k-2$ this is so if $\# \OF^{\times} \mid 2(k-2)$ as we assumed at the outset. 
Moreover, under our assumptions the pairing $[\cdot, \cdot]_m$ is perfect and we get an isomorphism $L(k-2, A)^0 \cong L(k-2,A)$ of $\Gamma$-modules. %[Indeed, we get an isomorphism $\phi: L \xrightarrow{\sim} L^{\vee}$ given by $\phi(l) = [l,\cdot]_m$, and we have for $x \in L$, $\phi(\gamma l) (x)= [\gamma l,x] = [\gamma l, \gamma \gamma^{-1} x] = [l, \gamma^{-1} x] = \phi(l) (\gamma^{-1}x) = (\gamma \cdot \phi(l))(x)$, because $\Gamma$ acts on $L^{\vee}$ via $(\gamma\cdot \psi)(x) = \psi(\gamma^{-1}x)$.] 
In particular %if $\# \OF^{\times} \mid 2(k-2)$ then 
in all the statements below one can replace $\tilde{L}(k-2, A)^0$ with $\tilde{L}(k-2,A)$.

\begin{thm}[Urban, \cite{Urban98}, Th\'eor\`eme 2.5.1] \label{2.5.1} 
%Suppose $\ell >k-2$, $\ell \nmid \Phi(\fN\fp)$. Then t
%\com{Added:}
If $\ell$ is a prime such that $\ell > k-1$ and $\ell \nmid N(\fN\fp)D_F\#\OF^{\times}$  then 
%\com{end added} 
there exists a pairing (induced by the cup product):
$$J_n: H^1_!(Y_n, \tilde{L}(k-2,\Oo)) \otimes H^2_!(Y_n, \tilde{L}(k-2,\Oo)^0) \to H^3_c(Y_n, \Oo) \cong \Oo,$$ which is perfect modulo torsion.  \end{thm}

\begin{lemma} \label{tfH1} If $A \subset \bfC$, %\com{Added:}
and $\ell$ is any prime such that $\ell >k-2$ 
%\com{end added} 
then $H^1_!(Y_n, \tilde{L}(k-2,A))$ is torsion-free. \end{lemma}
\begin{proof} The exact sequence $0 \to L(k-2,A) \to L(k-2,\bfC) \to L(k-2,\bfC/A) \to 0$ induces a long exact sequence in (group) cohomology $$\bigoplus_j H^0(\Gamma_{n,j}, L(k-2,\bfC/A)) \to \bigoplus_j  H^1(\Gamma_{n,j},L(k-2, A)) \to \bigoplus_j  H^1(\Gamma_{n,j}, L(k-2,\bfC)).$$ To prove the claim it is enough to show that that the sequence of $H^0$-groups is short exact. This follows from the assumption $\ell > k-2$ (cf. also \cite{SengunTurkelli09}, Lemma 6.2).
%$H^0(\Gamma_{n,j}, L(k-2,\bfC/A))=0$, but this is clear since $k\geq 3$. 
%\com{Check!}
\end{proof}

\subsection{Relation to automorphic forms} \label{Eichler 2}

 For an abelian group $M$ we set $M^{\rm tf} = M/{\rm torsion}$. 

 Let $K \subset \GL_2(\AFf)$ be an open compact subgroup.
We will write $\mS_{k}(K)$ for the $\bfC$-space of cuspidal automorphic forms on $\GL_2(\AF)$ of (parallel) weight $k$ and level $K_n$. We will not need the precise definition in what follows, but we refer the reader to \cite{Urban98}, section 3 for details. 

 \begin{thm} [Eichler-Shimura-Harder isomorphism] \label{ESH} There exist Hecke-equivariant isomorphisms (cf. e.g.\cite{Urban98}, section 3):
$$\delta_q : \mS_{k}(K) \xrightarrow{\sim} H^q_!(Y_K, \tilde{L}(k-2,\bfC)), \quad q\in \{1,2\}.$$
\end{thm}

Following Urban (\cite{Urban98}, section 4) for a fixed $k$ and a subring $A \subset \bfC$ we define the following two sets of lattices:
$$L^1_A(K_n) = \delta_1^{-1}(\iota_1(H^1_!(Y_n, \tilde{L}(k-2,A)))^{\rm tf});$$
$$L^2_A(K_n) = \delta_2^{-1}(\iota_2(H^2_!(Y_n, \tilde{L}(k-2,A)^0)^{\rm tf})),$$ where $\iota_q$ are the canonical maps on cohomology induced by the embedding $A \hookrightarrow \bfC$.

\begin{lemma} \label{basech} Let $A$ be a subring of $\bfC$. The lattices $L^q_A(K_n)$ satisfy the following properties:
\begin{itemize}
\item[(i)] one has $L^q_{\bfZ}(K_n) \otimes_{\bfZ}A \cong L^q_{A}(K_n)$ for $q=1,2$;
\item[(ii)] the lattices $L^1_A(K_n)$ and $L^2_A(K_n)$ are free $A$-modules of the same rank;
\item[(iii)] the lattices $L^1_A(K_n)$ and $L^2_A(K_n)$ are both stable under the action of the Hecke algebra $\bfT_{n,\bfZ}$.
\end{itemize}
\end{lemma}
\begin{proof} The first 
 assertion is Proposition 4.1.1 in \cite{Urban98}. Setting $A=\bfC$ in the (i) and using the fact that $L^q_{\bfC}(K_n) \cong H^q_!(X_n, \tilde{L}(k-2,\bfC)) \cong \mS_k(K_n)$ gives (ii). Part (iii) follows from the fact that $H^q_!(X_n, \tilde{L}(k-2,A))$ is Hecke-stable and the isomorphisms $\delta_q$ are Hecke-equivariant.
\end{proof}

\begin{cor} \label{inje} The maps $\iota_q$ are injective for $q=1,2$. \end{cor}
\begin{proof} This follows immediately from Theorem \ref{ESH} and Lemma \ref{basech}(i). \end{proof}

\begin{cor}\label{lat} %Suppose $\ell \nmid \Phi(\fN\fp)$. Then t
The pairing $J_n$ induces a perfect pairing $$\left<\cdot, \cdot\right>_n: L^1_{\Oo}(K_n) \otimes L^2_{\Oo}(K_n) \to \Oo, \quad \left< f, h \right>_{n} = J_n(\delta_1(f), \delta_2(h))$$ \end{cor}
\begin{proof} This follows from perfectness (mod torsion) of $J_n$ and the fact that by Corollary \ref{inje}, the maps $\delta_q$ induce isomorphisms between $H^q_!(X_n, L(k-2,\Oo))^{\rm tf}$ and the lattices $L^q_{\Oo}(K_n)$. \end{proof}

\section{Ihara's lemma for $H^2$} \label{sectionI1}
In this section we will prove an analogue of Ihara's lemma for the groups $H^2_!(Y_K, \tilde{L}(k-2,E/\Oo))$. 
More precisely, the goal of this section is to prove the following theorem.

\begin{thm} \label{IharaH2}
%\com{Remove: Assume  $\ell\nmid\#\OF^{\times}\Phi(\fN\fp)$.}
%\com{added:} 
Assume $\ell \nmid \#\OF^{\times}$. %\com{end added} 
 The map $$\alpha_2: \bigoplus_{j=1}^{\# \Cl_F} H^2_P(\Gamma_{0,j}, L(k-2,E/\Oo))^2 \to \bigoplus_{j=1}^{\# \Cl_F} H^2_P(\Gamma_{1,j},L(k-2,E/\Oo))$$ defined as $(f,g) \mapsto \alpha_1^{0,1}f + \alpha_1^{0,1\fp}\alpha_{\fp}^0g$ is injective. \end{thm}

\begin{proof} 
%[Proof of Theorem \ref{IharaH2}]
%Let $\Oo_{(\fp)}$ 
%denote the ring of $\fp$-integers in $F$, i.e., the
%elements of $F$, whose $\fq$-adic valuation is non-negative for every
%prime $\fq \neq \fp$.   For $j=1,2,\dots,h_F$ put 
%$$\SL_2(\Oo_{(\fp)})_j:=\left\{\bmat a&b\\c&d \emat \in \SL_2(F) \mid a,d \in \Oo_{(\fp)}, 
%b\in \fp_j \Oo_{(\fp)}, c\in \fp_j^{-1}\Oo_{(\fp)}\right\}.$$ 
 Recall that $\alpha_1^{*, *}$ are restriction maps.  Set $\Gamma'_{i,j}:= \Gamma_{i,j}\cap \SL_2(F)$ for $i=-1,0,1$.
We have a commutative diagram 
$$\xymatrix{\bigoplus_j H^2_P(\Gamma_{0,j},L(k-2,E/\Oo))^2 \ar[r]^{\alpha_2}\ar[d]^{\textup{res}} & \bigoplus_j H^2_P(\Gamma_{1,j},L(k-2,E/\Oo))\ar[d]\\ \bigoplus_j H^2_P(\Gamma'_{0,j},L(k-2,E/\Oo))^2 \ar[r]^{\alpha_2} &\bigoplus_j  H^2_P(\Gamma'_{1,j},L(k-2,E/\Oo)),}$$  The injectivity of the left vertical arrow follows from the fact that $\ell \nmid \#\OF^{\times}$. So, it is enough to prove that `bottom' $\alpha_2$ is injective.
The main ingredient in this proof is a result of Blasius, Franke and Grunewald which we now state in a special form pertaining to our situation. 

\begin{thm}[Blasius, Franke, Grunewald] \label{BFG} Let $\Gamma \subset \SL_2(F)$ be a congruence subgroup and let $\Gamma_{S} \supset \Gamma$ be an $S$-arithmetic subgroup of $\SL_2(F)$ with $S$ a finite set of finite primes of $F$. Then the image of the restriction map \be \label{G2}H^*(\Gamma_{S}, E) \to H^*(\Gamma,E)\ee coincides with the image of the space of $\SL_2(\bfC)$-invariant forms on the symmetric space $\SL_2(\bfC)/K_{\infty}$ in the de Rham cohomology of the locally symmetric space $\Gamma \setminus \SL_2(\bfC)/K_{\infty}$ (which is identified with the group cohomology of $\Gamma$). If $E$ is replaced by a non-constant irreducible $E$-representation, then the map (\ref{G2}) is the zero map.  \end{thm}

\begin{proof} This is Theorem 4 in \cite{BlasiusFrankeGrunewald94}.
% together with Theorem 5 in [loc.cit.] which guarantees that the cohomology initially defined with $\bfC$-coefficients is indeed the complexification of the cohomology with rational coefficients (hence our statement over $E$). 
\end{proof}

%\com{Added:}
%\begin{rem} The result of Blasius, Franke and Grunewald (Theorem 4 in \cite{BlasiusFrankeGrunewald94}) applies in much greater generality. In particular $\SL_2(F)$ may be replaced by any simply connected simple algebraic group $\mG$ defined over a number field under the assumption that there exists a place $v \in S$ such that $\mG(F_v)$ is not compact. \end{rem}
%\com{end added}

It follows from a theorem of Serre (cf. \cite{Klosin08}, Theorem 8) that $\Gamma'_{-1, j}$ is the amalgamated product of $\Gamma'_{0,j}$ and $(\Gamma'_{0,j})^{\fp}$ along $\Gamma'_{1,j}:=\Gamma'_{0,j} \cap (\Gamma'_{0,j})^{\fp}$. 
Hence using the exact cohomology sequence of Lyndon (cf. \cite{Serre03}, p. 127) one gets that the top row in the following diagram is exact (for any coefficients $M$ which we suppress):
\be \xymatrix@C9em@R5em{\bigoplus_j H^2(\Gamma'_{-1,j})\ar[r]^(.4){f \mapsto (f|_{\Gamma'_{0,j}},f|_{(\Gamma'_{0,j})^{\fp}})}\ar@{=}[d] &\bigoplus_j( H^2(\Gamma'_{0,j}) \oplus H^2((\Gamma'_{0,j})^{\fp})) \ar[r]^(.6){(f,g) \mapsto f|_{\Gamma'_{1,j}}-g|_{\Gamma'_{1,j}}}
\ar[d]_{\wr}^{(f,g) \mapsto (f, -\bsmat \pi \\ & 1 \esmat^* g)} & \bigoplus_j H^2(\Gamma'_{1,j})\ar@{=}[d]\\
\bigoplus_j H^2(\Gamma'_{-1,j}) \ar[r]^{\beta_M}_{f\mapsto (f|_{\Gamma'_{0,j}}, -\bsmat \pi \\ & 1 \esmat^* f|_{\Gamma'_{0,j}})}& \bigoplus_j H^2(\Gamma'_{0,j})^2 \ar[r]^{\alpha_M}_{(f,g) \mapsto (f|_{\Gamma'_{1,j}} + \bsmat \pi^{-1} \\ & 1 \esmat^* g |_{\Gamma'_{1,j}})} & \bigoplus_j H^2(\Gamma'_{1,j}),}\ee 
where the map $\bsmat \pi \\ & 1\esmat^*$ is induced from the map $H^2(Y_0^{\fp},\tilde{M}_0^{\fp}) \to H^2(Y_0, \tilde{M}_0)$ arising from the isomorphism $K_0^{\fp} \xrightarrow{\sim} K_0$ given by conjugation by $\bsmat \pi \\ & 1\esmat$ (i.e., equals the inverse of $\alpha_{\fp}^0$). Note that the maps in the bottom row and the map represented by the middle vertical arrow do not necessarily preserve the direct summands (they do if $\fp$ is principal), but the maps in the top row do.
%$(\gamma^* g)(x):= g(\gamma^{-1} x \gamma)$. 

It is clear that this diagram commutes (essentially by the definitions of the maps involved). 
%We now claim that the conjugation by $\bmat \pi \\ & 1 \emat$ induces an isomorphism between $H^2((\Gamma'_{0,j})^{\fp})$ and $H^2(\Gamma_{0,j})$. \com{Prove! - note it mixes the factors for different $j$'s!} 
Let $(f,g) \in \ker \alpha_M$. Then the corresponding element in $\bigoplus_j (H^2(\Gamma'_{0,j}) \oplus H^2((\Gamma'_{0,j})^{\fp}))$ is in the image of the top left arrow. Hence by commutativity $(f,g)$ is in the image of $\beta_M$, so, to prove the theorem it is enough to show that $\beta_{L(k-2,E/\Oo)}=0$. For this it suffices to prove that both of the restriction maps $f \mapsto f|_{\Gamma'_{0,j}}$ and $f \mapsto f|_{(\Gamma'_{0,j})^{\fp}}$ are the zero maps (when $M=L(k-2,E/\Oo)$). We will show that the first restriction map is zero, the proof of the vanishing of the second map being essentially identical.

%First note that by Assumption \ref{tor1} the third arrow in 
%$$H^2(\Gamma'_{-1,j},\Oo)\to H^2(\Gamma'_{-1,j},E) \to H^2(\Gamma'_{-1,j},E/\Oo)\to H^3(\Gamma'_{-1,j}, \Oo)$$ is zero, so the second arrow is surjective.
 Consider the following commutative diagram with exact columns and horizontal arrows being restriction maps:
\be \label{EEO} \xymatrix{H^2(\Gamma'_{-1,j},L(k-2,E)) \ar[r] \ar[d] & H^2(\Gamma'_{0,j},L(k-2,E))\ar[d] \\
H^2(\Gamma'_{-1,j},L(k-2,E/\Oo)) \ar[r]\ar[d] & H^2(\Gamma'_{0,j},L(k-2,E/\Oo))\ar[d]\\
H^3(\Gamma'_{-1,j},L(k-2,\Oo)) \ar[r] & H^3(\Gamma'_{0,j},L(k-2,\Oo))
}\ee

By Theorem \ref{BFG} (with $S=\{\fp\}$) the top horizontal arrow in (\ref{EEO}) is the zero map (for $k>2$ the representation $L(k-2,E)$ is irreducible (cf. e.g. \cite{Harder87}, p.45) and for $k=2$ one checks that the space of $\SL_2(\bfC)$-invariant forms is zero). Since $\Gamma_{0,j}$ is a torsion-free discrete subgroup of $\GL_2(F)$, it follows from Proposition 18(b) of \cite{Serre71}, that the cohomological dimension of $\Gamma_{0,j}$ is no greater than two (because the real dimension of the symmetric space for $\GL_2(F)$ is 3). As the index of $\Gamma'_{0,j}$ in $\Gamma_{0,j}$ is finite, Proposition 5(b) in [loc.cit.] implies that the cohomological dimension of $\Gamma'_{0,j}$ is also no greater than two. Hence we must have $H^3(\Gamma'_{0,j},L(k-2,\Oo))=0$. Thus the bottom horizontal arrow in diagram (\ref{EEO}) is the zero map. We will now show that the middle horizontal arrow is also zero. 

Write $A$ for the image of the top-left vertical arrow, $B$ for $H^2(\Gamma'_{-1,j},L(k-2,E/\Oo))$ and $C$ for the image of the bottom-left vertical arrow. We have a short exact sequence \be \label{ABC} 0 \to A \to B \to C \to 0.\ee It is enough to show that (\ref{ABC}) splits as a sequence of $\Oo$-modules. By taking the Pontryagin duals we obtain an exact sequence \be \label{ABC2} 0 \to \Hom_{\Oo}(C, E/\Oo) \to \Hom_{\Oo}(B,E/\Oo)\to \Hom_{\Oo}(A, E/\Oo) .\ee 
Let us first show that all of the $\Hom$-groups in (\ref{ABC2}) are finitely generated $\Oo$-modules. 

As noted before, $\Gamma'_{-1, j}$ is the amalgamated product of $\Gamma'_{0,j}$ and $(\Gamma'_{0,j})^{\fp}$ along $\Gamma'_{1,j}$. By the same argument as in the case of $\Gamma'_{0,j}$ one sees that $(\Gamma'_{0,j})^{\fp}$ and $\Gamma'_{1,j}$ all have cohomological dimension no greater than 2. Hence, by Proposition 7 of \cite{Serre71}, the group $\Gamma'_{-1, j}$ has cohomological dimension no greater than 3. Thus, one can apply Proposition 4 of [loc.cit.] to conclude that $H^3(\Gamma'_{-1,j},L(k-2,\Oo))$ is a finitely generated $\Oo$-module. Thus $C$ and $\Hom(C, E/\Oo)$ are finite groups. Moreover, by Remarque 1 of [loc.cit.] we also get that $H^2(\Gamma'_{-1,j},L(k-2,E))$ is a finite dimensional vector space over $E$.  It follows that $A \cong (E/\Oo)^m$ for some $m$. 
%[Indeed, we have $A \cong H^2(E)/\ker(H^2(E) \to H^2(E/\Oo)) \cong H^2(E)/Im (H^2(\Oo) \to H^2(E))$, and since $H^2(\Oo)$ is a finitely generated $\Oo$-module, we get the last image to be $\Oo^s$. We must have $s=m$, because $H^2(E/\Oo)$ is torsion.]
 Thus, $\Hom_{\Oo}(A, E/\Oo) \cong \Oo^m$.  Hence all the Hom-groups in (\ref{ABC2}) are finitely generated $\Oo$-modules. So, in particular $\Hom_{\Oo}(B, E/\Oo) \cong \Oo^m \oplus G$, where $G$ is a finite group. 
%Since $B$ is an $\Oo$-module 
%its double Pontryagin dual is equal to $B$ (by \cite{Rubin}, Introduction) \com{delete this reference in published version}, and 
We conclude that $B \cong (E/\Oo)^m \oplus G$. 
Then it is clear that 
%Note that $C$ must be a finite group. Since the image of a divisible element must be divisible, we get that the $A\cap G=0$ and since $(E/\Oo)^m$ has no proper divisible submodules other than $(E/\Oo)^s$ for $s <n$ we must have that $A$ as a submodule of $B$ must equal the direct summand $(E/\Oo)^m$. Thus 
(\ref{ABC}) splits. \end{proof}

\section{Ihara's lemma for $H^1$} \label{sectionI2}
In this section we will prove an analogue of Ihara's lemma for the groups $H^1_!(Y_K, \tilde{L}(k-2,E/\Oo))$. 
More precisely, the goal of this section is to prove the following theorem.

\begin{thm} \label{IharaH1}
 Suppose $\ell >k-2$ and $\ell \nmid \# \OF^{\times}\Phi(\fN)$. The map $$\alpha_1: \bigoplus_{j=1}^{\# \Cl_F} H^1_P(\Gamma_{0,j}, L(k-2,E/\Oo))^2 \to \bigoplus_{j=1}^{\# \Cl_F} H^1_P(\Gamma_{1,j},L(k-2,E/\Oo))$$ defined as $(f,g) \mapsto \alpha_1^{0,1}f + \alpha_1^{0,1\fp}\alpha_{\fp}^0g$ is injective. \end{thm}
\begin{rem} In \cite{Klosin08} the author proved Theorem \ref{IharaH1} for the case $k=2$, so we will assume below that $k>2$.  \end{rem}
\begin{rem} If one knew that $H^2_P(\Gamma_{0,j},L(k-2,\Oo))$ was torsion-free then the proof of Theorem \ref{IharaH2} would carry over verbatim  to this case (with $n$-degree cohomology replaced by $(n-1)$-degree cohomology) as then the bottom horizontal arrow in diagram (\ref{EEO}) would have to be zero on the maximal torsion subgroup which in turn contains the image of the bottom-left vertical arrow. In this case the assumptions that $\ell > k-2$ and $\ell \nmid \Phi(\fN)$ are unnecessary. \end{rem}

\begin{proof}[Proof of Theorem \ref{IharaH1}] In this proof we mostly follow Diamond \cite{Diamond91}, proof of Lemma 3.2, but indicate where the arguments of [loc.cit.] need to be modified. 
Let $\Oo_{(\fp)}$ 
denote the ring of $\fp$-integers in $F$, i.e., the
elements of $F$, whose $\fq$-adic valuation is non-negative for every
prime $\fq \neq \fp$.   For $j=1,2,\dots,\#\Cl_F$ put 
$$\SL_2(\Oo_{(\fp)})_j:=\left\{\bmat a&b\\c&d \emat \in \SL_2(F) \mid a,d \in \Oo_{(\fp)}, 
b\in \fp_j \Oo_{(\fp)}, c\in \fp_j^{-1}\Oo_{(\fp)}\right\}.$$ 
Define the $j$-th principal congruence 
subgroup of level $\fN$ by $$\G_{\fN, j}:= \left\{\bmat a&b\\c&d \emat \in \SL_2(\Oo_{(\fp)})_j \mid
b,c\in \fN \Oo_{(\fp)}, a\equiv d \equiv 1 \hf \textup{mod} \hs \fN 
\Oo_{(\fp)}
\right\}.$$
By Lemma 10 in \cite{Klosin08}, we may assume that the ideals $\fp_j$ satisfy $(N\fp_j-1, \ell)=1$.
 We have a commutative diagram 
$$\xymatrix{\bigoplus_j H^1_P(\Gamma_{0,j},L(k-2,E/\Oo))^2 \ar[r]^{\alpha}\ar[d]^{\textup{res}} & \bigoplus_j H^1_P(\Gamma_{1,j},L(k-2,E/\Oo))\ar[d]\\ \bigoplus_j H^1_P(\Gamma_{j}(\fN),L(k-2,E/\Oo))^2 \ar[r]^{\alpha} &\bigoplus_j  H^1_P(\Gamma_{\cap},L(k-2,E/\Oo)),}$$ where the group $\Gamma_j(\fN)\subset \Gamma_{0,j}\cap \SL_2(\OF)_j$ is defined as the principal congruence subgroup of level $\fN$, $\Gamma_{\cap}:= \Gamma_j(\fN) \cap \Gamma_j(\fN)^{\fp}$ and %$\Gamma_j(\fN\fp)\subset \Gamma_{1,j}\cap \SL_2(\OF)_j$ are  (resp. $\fN\fp$) and 
$\SL_2(\OF)_j$ is defined in the same way as $\SL_2(\Oo_{(\fp)})_j$ but with $\Oo_{(\fp)}$ replaced by $\OF$. The injectivity of the left vertical arrow follows from the fact that $\ell \nmid \#\OF^{\times}\Phi(\fN)$. 

 As in \cite{Diamond91} we will reduce the problem to showing that $H^1_{S_j}(\Gamma_{\fN,j},L(k-2,E/\Oo))=0$ for all $j$, where $S_j \subset \Gamma_{\fN,j}$ is the subset of elements conjugate in $\SL_2(\Oo_{(\fp)})_j$ to $\bmat 1&* \\ 0&1\emat$. %\com{Diamond has $\pm \bmat 1&* \\ 0&1\emat$ but we assume that $N>2$, so $-1$ is not congruent to 1 mod $\fN$. Since conjugation does not alter trace no matrix in $\Gamma_{\fN,j}$ can be conjugated to an upper-triangular matrix with $-1$ on the diagonal.}
Here for a group $G$, a subset $Q \subset G$ and a $G$-module $A$ we write $H^1_Q(G,A)$ for the subgroup of $H^1(G,A)$ whose classes are represented by  the cocycles $u$ satisfying $u(\gamma) \in (\gamma-1)A$ for all $\gamma \in Q$ (by which we mean that for every $\gamma \in G$ there exists $a \in A$ such that $u(\gamma)=\gamma a -a$). Note that in the case when $G=\Gamma_j(\fN)$ and $Q$ is the subset of parabolic elements one clearly has  $H^1_Q(G,A)\supset H^1_P(G,A)$.
% (in fact it follows easily that this is an equality).
% but as opposed to the situation studied in \cite{Diamond91} one does not get equality because the Borels of $\Gamma_{\fN,j}$ are not cyclic, so in our case the condition of reducing to a coboundary on a Borel (i.e., one has $c(\gamma)=\gamma a -a$ with the \emph{same} $a$ for all $\gamma$ in a given Borel) is stronger than the condition defining $H^1_Q$.
% equals the parabolic cohomology $H^1_P(G,A)$. [Indeed, by definition $H^1_P$ consists of the cohomology classes $c$ vanishing on the elements of $Q$, i.e., such that there exists $a \in A$ (depending on $c$) with $c(\gamma) = (\gamma-1)a$ for all $\gamma \in G$, hence clearly $H^1_P \subset H^1_Q$. However, since every Borel subgroup of $\Gamma_j(\fN)$ is a cyclic group (since $-1 \not\in \Gamma_j(\fN)$) we see that if $c(\gamma)=(\gamma-1)a$ for a generator $\gamma$ of $\Gamma_j(\fN)$ and some $a$ (depending on the choice of $\gamma$), then $c(\gamma^n) = (\gamma^n-1)a$, hence in fact $H^1_Q \subset H^1_P$.]

%and $$H^1_{S_j}(\Gamma, M):= \ker(H^1(\Gamma, M) \to \prod_{g \in \SL_2(\Oo_{(\fp)})_j} H^1(gJ_j g^{-1}\cap \Gamma, M),$$ where $J_j$ is the subgroup of $\Gamma_{\fN,j}$ consisting of matrices of the form $\bmat 1 &* \\ 0 & 1\emat$.

By arguments analogous to the ones used in the proof of Theorem \ref{IharaH2} one obtains a commutative diagram 
\be \xymatrix@C9em@R5em{\bigoplus_j H^1(\Gamma_{\fN,j})\ar[r]^(.4){f \mapsto (f|_{\Gamma_j(\fN)},f|_{\Gamma_j(\fN)^{\fp}})}\ar@{=}[d] &\bigoplus_j( H^1(\Gamma_j(\fN)) \oplus H^1(\Gamma_j(\fN)^{\fp})) \ar[r]^(.6){(f,g) \mapsto f|_{\Gamma_{\cap}}-g|_{\Gamma_{\cap}}}
\ar[d]^{(f,g) \mapsto (f, -\bsmat \pi \\ & 1 \esmat^* g)}_{\wr} & \bigoplus_j H^1(\Gamma_{\cap})\ar@{=}[d]\\
\bigoplus_j H^1(\Gamma_{\fN,j}) \ar[r]^{\beta}_{f\mapsto (f|_{\Gamma_j(\fN)}, -\bsmat \pi \\ & 1 \esmat^* g|_{\Gamma_j(\fN)})}& \bigoplus_j H^1(\Gamma_j(\fN))^2 \ar[r]^{\alpha}_{(f,g) \mapsto (f|_{\Gamma_{\cap}} + \bsmat \pi^{-1} \\ & 1 \esmat^* g |_{\Gamma_{\cap}})} & \bigoplus_j H^1(\Gamma_{\cap})}\ee with exact rows.

\begin{lemma}\label{coc}  Every cocycle $u$ in $Z^1(\Gamma_{\fN, j}, L(k-2,E/\Oo))$, which is mapped by $\beta$ into $\bigoplus_j H^1_P(\Gamma_j(\fN))^2$  represents a class in $H^1_{S_j}(\Gamma_{\fN,j}, L(k-2,E/\Oo))$. \end{lemma}
\begin{proof} This is proved as in \cite{Diamond91}, p.211 using the fact that $H^1_P \subset H^1_Q$.
%Let us first show that if $\sigma \in S_j$, then  $\sigma^{p^M} \in \Gamma_j(\fN)$ for a sufficiently large $M$. Write $\sigma = \bmat a&b\\ c&d \emat \bmat 1&x \\ & 1 \emat \bmat a&b \\ c& d \emat^{-1}$ with $\bmat a&b\\ c&d \emat  \in \SL_2(\Oo_{(\fp)})_j$. Then $$\sigma^{p^M} =  \bmat a&b\\ c&d \emat \bmat 1&p^Mx \\ & 1 \emat \bmat a&b \\ c& d \emat^{-1} = \bmat 1-acp^Mx & a^2p^Mx \\ -c^2p^Mx & 1+acp^Mx\emat.$$  So, we see that by choosing $M$ large enough we can clear the denominators in the entries of $\sigma^{p^M}$.] 
% One  has $$u(\sigma^{p^M}) = (1+\sigma + \sigma^2 +\dots+ \sigma^{p^M-1}) u(\sigma).$$ Now suppose that $\beta(u)$ represents a class in $H^1_P(\Gamma_j(\fN))^2$. Since $\sigma^{p^M} \in \Gamma_j(\fN)$, and $u$ restricted to $\Gamma_j(\fN)$ lies in $H^1_P\subset H^1_Q$, we know that we must have  $u(\sigma^{p^M}) \in (\sigma^{p^M}-1) L(k-2, E/\Oo)$. Since $(1+\sigma + \sigma^2 +\dots+ \sigma^{p^M-1})$ is an automorphism of $L(k-2, E/\Oo)$, we see that the cohomology class given by $u$ lies in $H^1_{S_j}(\Gamma_{\fN,j}, L(k-2, E/\Oo))$. 
\end{proof}

So, it is enough to prove that $H^1_{S_j}(\Gamma_{\fN,j}, L(k-2, E/\Oo))=0$. In fact one can easily see (cf. \cite{Diamond91}, p. 211) that it is enough to prove that $H^1_{S_j}(\Gamma_{\fN,j}, L(k-2, \bfF))=0$, where $\bfF=\Oo/\varpi \Oo$. The inflation-restriction exact sequence
$$0 \to H^1(\SL_2(\OF/\ell), L(k-2, \bfF)) \to H^1(\Gamma_{\fN,j}, L(k-2, \bfF)) \to H^1(\Gamma_{\fN\ell, j}, L(k-2, \bfF))$$ gives rise to an exact sequence \be \label{D1} 0 \to H^1_Q(\SL_2(\OF/\ell), L(k-2, \bfF)) \to H^1_{S_j}(\Gamma_{\fN,j}, L(k-2, \bfF)) \to H^1_{S_j \cap \Gamma_{\fN\ell, j}}(\Gamma_{\fN\ell, j}, L(k-2, \bfF)),\ee where $Q$ is an $\ell$-Sylow subgroup of $\SL_2(\OF/\ell)$. 

\begin{lemma}\label{ST} One has $H^1_Q(\SL_2(\OF/\ell), L(k-2, \bfF))=0$.  \end{lemma}
\begin{proof} This is just an adaptation of the proof of Lemma 3.1 in \cite{Diamond91}. 
%In fact in the case when $\ell$ is a split prime this was proved by Seng\"un and T\"urkelli \cite{SengunTurkelli09} - see the proof of Theorem 7.1 - our claim follows from the injectivity of their map $H^1(\Gamma, E_{r,s}) \to H^1(\Gamma, I_{r,s})$ by using the inflation-restriction sequence for the level $\ell$ principal congruence subgroup of $\SL_2(\OF)$). 
Let us outline the arguments for the case of an inert $\ell$ - the split case is proved similarly. Set $G=\SL_2(\OF/\ell)$ and write $I$ for the representation $\Ind_Q^G (\bfF)$, where we treat $\bfF$ as a trivial $Q$-module. We can regard elements of $I$ as $\bfF$-valued functions on $\bfF_{\ell^2} \times \bfF_{\ell^2} \setminus \{(0,0)\}$. Set $\Sigma :=\{ \sigma :\bfF_{\ell^2} \hookrightarrow \bfF\}$ and write $\Sym^{k-2}_{\sigma}(\bfF)$ for the $\bfF_{\ell^2}[G]$-module of homogeneous polynomials in two variables with coefficients in $\bfF$ where the actions of both $\bfF_{\ell^2}$ and $G$ are via the embedding $\sigma$. While the action of $G$ on $I$ is canonical we can also add an $\bfF_{\ell^2}$-module structure via $\sigma$ and denote the resulting $\bfF_{\ell^2}[G]$-module by $I_{\sigma}$. For every $\sigma \in \Sigma$ we then have a natural $\bfF_{\ell^2}[G]$-equivariant injection $\phi_{\sigma}: \Sym^{k-2}_{\sigma}(\bfF) \to I_{\sigma}$
given by $\phi_{\sigma}(P)(a,b) = P(\sigma(a),\sigma(b))$. We then define an $\bfF_{\ell^2}[G]$-equivariant injection
$\phi:L(k-2, \bfF) \to \bigotimes_{\sigma \in \Sigma} I_{\sigma}$ by $\phi :=\bigotimes_{\sigma \in \Sigma} \phi_{\sigma}$. Let us fix an ordering of the elements of $\Sigma$ and denote them by $\sigma_1$ and $\sigma_2$. We will introduce a convention that for $\bfF$-modules $M$, $N$ in tensor products $M \otimes_{\bfF_{\ell^2}}N$ the $\bfF_{\ell}^2$-action on the module on the left of $\otimes_{\bfF_{\ell}^2}$ is via $\sigma_1$ and on the module on the right of $\otimes_{\bfF_{\ell}^2}$ the action is via $\sigma_2$. In particular we will just write  $I\otimes I$,  instead of $\bigotimes_{\sigma \in \Sigma} I_{\sigma}$.
Note that $I \otimes I \cong \Ind_Q^G(\bfF\otimes_{\bfF_{\ell^2}}\textup{Res}_Q^G I)$, where $\textup{Res}_Q^G$ denotes the restriction to $Q$ functor. It is then a consequence of Shapiro's lemma that $$H^1_Q(G, I \otimes I) = H^1_Q(Q, \bfF\otimes_{\bfF_{\ell^2}}\textup{Res}_Q^G I).$$ Since $Q$ is solvable, an easy calculation shows that the latter cohomology group is zero. 
Thus  we have reduced the problem to showing that the sequence:
\be \label{seq}0 \to L(k-2, \bfF)^G \to  (I\otimes I)^G \to \left(\frac{I\otimes I}{L(k-2,\bfF)}\right)^G \to 0\ee is exact. The rest of the proof consists of decomposing the module $I\otimes I$ and analyzing its $G$-fixed points. Here  we follow closely the strategy of the proof of Lemma 3.1 in \cite{Diamond91} with modifications needed to account for dealing with $I\otimes I$ instead of just $I$. In this we use some results of \cite{SengunTurkelli09} which dealt with a very similar situation.

We have $I_{\sigma}=\bigoplus_{n=0}^{\ell^2-2} I_{n,\sigma}$, where $$I_{n,\sigma}=\{f: \bfF_{\ell^2} \times \bfF_{\ell^2} \setminus \{(0,0)\}\to \bfF \mid f((xa,xb))=\sigma(x)^n f((a,b))\}.$$ Observe that every $I_{n,\sigma}$ is an $\bfF$-vector space of dimension $\ell^2+1$ and that $L(k-2, \bfF) \subset I_{k-2}\otimes I_{k-2}:=I_{\k-2,\sigma_1}\otimes_{\bfF_{\ell^2}}I_{k-2,\sigma_2}$. Moreover, note that the action of $G$ on $I\otimes I$ preserves every summand $I_m \otimes I_n$, so all we need to prove is that the map $$(I_{k-2}\otimes I_{k-2})^G \to \left(\frac{I_{k-2}\otimes I_{k-2}}{L(k-2, \bfF)}\right)^G$$ is surjective. In fact we will prove that the module on the right is zero. To do this we decompose $I_{k-2}\otimes I_{k-2}$ as in 
 Lemma 5.5 of \cite{SengunTurkelli09} (with obvious modifications), and thus it suffices to show that \be \label{eq54} (\Sym_{\sigma_1}^{\ell^2-1-(k-2)}(\bfF)) \otimes I_{k-2})^G\oplus (I_{k-2}\otimes \Sym_{\sigma_2}^{\ell^2-1-(k-2)} (\bfF))^G=0.\ee
Write $\bfF^{r,\sigma}$ for the one-dimensional $\bfF$-vector space on which $\bmat a&b\\c&d \emat \in G$ acts via $\sigma(d)^r$. Note that we have an $\bfF_{\ell^2}[G]$-module isomorphism: \begin{multline} I_{r,\sigma} \cong \Ind_P^G(\bfF^{r,\sigma})\\
= \left\{ f: G \to \bfF \mid f\left(\bmat d^{-1} & b \\ 0 & d \emat g\right) = \sigma(d)^r f(g) \hs\textup{for every}\hs g\in G, \bmat d^{-1} & b \\ 0 & d \emat\in P\right\},\end{multline} where $P\subset G$ is the upper-triangular Borel subgroup of $G$.
One then has \be \label{eq76}I_{k-2,\sigma_1}\otimes \Sym^{\ell^2-1-(k-2)}_{\sigma_2}(\bfF) \cong \Ind_P^G(\bfF^{k-2, \sigma_1}\otimes \textup{Res}_P^G  \Sym^{\ell^2-1-(k-2)}_{\sigma_2}(\bfF) ).\ee It is easy to check that $(\bfF^{k-2, \sigma_1}\otimes\Sym^{\ell^2-1-(k-2)}_{\sigma_2}(\bfF))^P=0$ or our range of $k$. Then by Shapiro's Lemma and (\ref{eq76}) we obtain that the second direct summand in (\ref{eq54}) is zero. We prove that the first one is zero in an analogous way. \end{proof}

%where $E_r^s := \det^s \otimes_{\bfF_{\ell^2}} L(r,\bfF)$

It now suffices to prove the vanishing of $H^1_{S_j \cap \Gamma_{\fN\ell, j}}(\Gamma_{\fN\ell, j}, L(k-2, \bfF)).$
For this first note that $L(k-2, \bfF)$ is a trivial $\Gamma_{\fN\ell, j}$-module, so one has $$ H^1_{S_j \cap \Gamma_{\fN\ell, j}}(\Gamma_{\fN\ell, j}, L(k-2, \bfF)) = \ker (\Hom(\Gamma_{\fN\ell, j}, L(k-2, \bfF)) \to \Hom(U_{\fN\ell,j}, L(k-2, \bfF))),$$ where $U_{\fN\ell,j}$ is the smallest normal subgroup of $\SL_2(\Oo_{(\fp)})_j$ containing the matrices  of the form $\bmat 1 & x \\ & 1\emat \in \Gamma_{\fN\ell,j}$. This means that $x \in \fN\ell\fp_j \Oo_{(\fp)}$. 
Note that since $\Gamma_{\fN\ell,j}$ is normal in $\SL_2(\Oo_{(\fp)})$, the group $U_{\fN\ell}$ is indeed contained in $\Gamma_{\fN\ell,j}$. 

By a result of Serre (cf. \cite{Serre70}, Theoreme 2(b), p.498 or \cite{Klosin08}, Theorem 11) we have that $U_{\fN\ell,j} = \Gamma_{\fN\ell\fp_j\Oo_{(\fp)}}$. We now argue as in the proof of Proposition 4 of \cite{Klosin08}. Put
$$\Gamma'_{\fN\ell\fp_j}:= \left\{\bmat a&b\\c&d \emat \in \SL_2(\Op)_j \mid
b,c\in \fN\ell\fp_j \Op
\right\}.$$ Let $f \in \Hom(\Gamma_{\fN\ell\fp_j\Oo_{(\fp)}}, L(k-2,\bfF))$.
Since $\Gamma_{\fN\ell\fp_j\Oo_{(\fp)}}$ is a normal subgroup of $\Gamma'_{\fN\ell\fp_j}$ of 
index $N\fp_j -1$ we have $f \left(\Gamma'_{\fN\ell\fp_i}\right)=0$ by our choice of the ideals $\fp_j$.
 On one 
hand $f$ is zero on the elements of the form $\bmat 1\\ c &1 \emat$, 
$c \in \fN \ell\fp_i^{-1} \Op$, and on the other hand elements of this form together with 
$\Gamma'_{\fN\ell\fp_j}$ generate $\Gamma_{\fN\ell,j}$, so $f(\Gamma_{\fN\ell,j})=0$, as 
asserted. This implies that $H^1_{S_j \cap \Gamma_{\fN\ell, j}}(\Gamma_{\fN\ell, j}, L(k-2, \bfF))$ vanishes. 
The claim now follows from exactness of (\ref{D1}).
\end{proof}

\section{The module of congruences} \label{The module of congruences}

In this section we will explain the relationship of Theorems \ref{IharaH2} and \ref{IharaH1} to the problem of the existence of level-raising congruences for the modules of automorphic forms over $F$ and the cohomology. 

In this section we keep the assumptions from section \ref{The Eichler 1}.
Let $M$ be a finitely generated free $\Oo$-module. For a submodule $N$ of $M$ we define $$N^{\rm sat} = (N \otimes_{\Oo}E) \cap M.$$ 
\begin{lemma} [Bella\"iche - Graftieaux, \cite{BellaicheGraftieaux06}, Lemme 4.1.1] \label{4.1.1} Let $u :N \to M$ be an injective homomorphism. For a sufficiently large positive integer $s$ one has an isomorphism of $\Oo$-modules:
$$u(N)^{\rm sat}/u(N) \cong \ker  u_s, \quad \textup{where $u_s:=u\otimes 1 : N \otimes \Oo/\varpi^s \to M\otimes \Oo/\varpi^s.$} $$ \end{lemma}
%One has $N=N^{\rm sat}$ if and only if $N$ is a direct factor of $M$. 
Let $A$ and $B$ be two submodules of $M$ such that $A\cap B=0$, $A^{\rm sat}=A$ and $B^{\rm sat}=B$.  
Following Bella\"iche and Graftieaux we define the module of congruences between $A$ and $B$ by $(A\oplus B)^{\rm sat}/(A\oplus B)$. %\com{Added:} 
%The study of modular congruences using such modules originates from the work of Hida (see \cite{Hida93} for an introduction to the subject). 
%\com{end added: add references here}

The maps $\alpha_q$ ($q=1,2$) induce corresponding (injective) maps 
 $$\alpha_q: L_{\Oo}^q(K_0)^2 \to L_{\Oo}^q(K_1) \quad (f,g) \mapsto [K_01K_1]f + [K_0\eta K_1]g,$$
where the double cosets act on $f$ and $g$ in the usual way.
Define $\alpha^+_1: L_{\Oo}^{2}(K_1) \to L_{\Oo}^2(K_0)^2$ to be the adjoint of $\alpha_1$ and $\alpha^+_2: L_{\Oo}^{1}(K_1) \to L_{\Oo}^1(K_0)^2$ to be the adjoint of $\alpha_2$ with respect to the pairings $\left<\left<\cdot, \cdot\right>\right>$ and $\left<\cdot, \cdot\right>_1$, where $\left<\left<\cdot, \cdot\right>\right>$ is the pairing   $L^1_{\Oo}(K_0)^2 \otimes L_{\Oo}^2(K_0)^2 \to \Oo$ defined by $\left<\left<(f,g), (f',g')\right>\right> = \left<f,f'\right>_0 + \left<g,g'\right>_0$. 
%Write $\alpha^+_{q,\bfC}$ for their complexifications. More precisely, $\alpha^+_{q, \bfC}$  is the adjoint of $\alpha_{q,\bfC}:= \alpha_q \otimes_{\Oo}\bfC$ with respect to the pairings $\left<\left< \cdot, \cdot \right>\right>$ and $\left< \cdot, \cdot \right>_1$, where $\left<\left< \cdot, \cdot \right>\right>$ is a pairing on  $L_{\bfC}^1(K_0)^2=L_{\bfC}^2(K_0)^2$ defined by $$\left<\left< (f,g), (f',g')\right>\right> :=\left< f, f' \right>_0 + \left< g, g' \right>_0.$$ 

 Set $M_q= L^q_{\Oo}(K_1)$, $A_q=\image (\alpha_q)$, $B_q=A_q^{\perp}\subset M_q$. For brevity write $\bfT_{\Oo}$ for $\bfT_{1, \Oo}$.  
One clearly has $\mS_k(K_1) = M_q\otimes_{\Oo}\bfC = X\oplus Y$, where $X:=(A_1\otimes_{\Oo}\bfC) = (A_2 \otimes_{\Oo}\bfC)$ and $Y:= (B_1 \otimes_{\bfZ}\bfC)=(B_2 \otimes_{\bfZ}\bfC)$ with $X$ and $Y$ stable under the action of $\bfT_{\Oo}$. Define $\bfT_{X}$ (resp. $\bfT_{Y}$) to be the image of $\bfT_{\Oo}$ inside $\End_{\bfC}(X)$ (resp. $\End_{\bfC}(Y)$). Then one has the following canonical inclusion (given by restrictions): $\bfT_{\Oo} \hookrightarrow \bfT_X \oplus \bfT_Y$. 

Let the (finite) quotient $C(\bfT_{\Oo}):=(\bfT_X\oplus \bfT_Y)/\bfT_{\Oo}$ be the Hecke congruence module between $X$ and $Y$. Note that if $C(\bfT_{\Oo}) \neq 0$ there exists Hecke eigenforms $f \in A_q$ and $g\in B_q$ whose eigenvalues (for all Hecke operators in $\bfT_{\Oo}$) are congruent modulo $\varpi$. 

\begin{lemma} \label{decomp1} Let $q\in \{1,2\}$. Suppose that $C(\bfT_{\Oo})=0$ and that $A_q=A_q^{\rm sat}$. Then $M_q^{\rm sat} =M_q=A_q \oplus B_q$, i.e., the module of congruences between $A_q$ and $B_q$ is zero. \end{lemma}
\begin{proof} This follows from \cite{Ghate02cong}, Lemma 4 combined with Lemma 1 together with the easy facts that $B_q=B_q^{\rm sat}$ and $M_q=M_q^{\rm sat}$.
\end{proof} 

Assume for the moment that $A_q=A_q^{\rm sat}$ for $q=1,2$ (this equality is sometimes referred to as `Ihara's lemma' - cf. \cite{BellaicheGraftieaux06}). It implies that there exists an $\Oo$-submodule $P_q\subset L_{\Oo}^q(K_1)$ such that $L_{\Oo}^q(K_1) = A_q \oplus P_q$. Given this,  for $n=0,1$ we can construct $\Oo$-module isomorphisms $\Psi_n: L^1_{\Oo}(K_n) \xrightarrow{\sim} L^2_{\Oo}(K_n)$ such that  $\Psi_1 = \Psi_A \oplus \Psi_P$ with $\Psi_A:A_1 \xrightarrow{\sim} A_2$ and $\Psi_P: P_1 \xrightarrow{\sim} P_2$, where $L^q_{\Oo}(K_1) =A_q \oplus P_q$.  Furthermore when we consider $\Psi_n$ as an automorphism of $L^q_{\bfC}(K_n)$ and $\Psi_A$ as an automorphism of $X$, then $\det \Psi_n$ and $\det \Psi_A$ are independent of the choice of $\Psi_n$ and $\Psi_A$ up to a unit in $\Oo$.

Suppose now in addition that $C(\bfT_{\Oo})=0$ (i.e., that there are no level raising congruences). 
Then Lemma \ref{decomp1} implies that $M_q=A_q\oplus B_q$ for $q=1,2$, i.e.,  we can write $\alpha_q$ as $(\alpha_A,0)$. Consider the following sequence of isomorphisms:
\be \label{diag90} L^1_{\Oo}(K_0)^2 \xrightarrow{\alpha_A}  A_1 \xrightarrow{\Psi_A} A_2 \xrightarrow{\alpha^+_A}  L^2_{\Oo}(K_0)^2 \xrightarrow{(\Psi_0, \Psi_0)^{-1}}  L^1_{\Oo}(K_0)^2\ee where surjectivity of $\alpha_A^+$ follows from perfectness of the pairings involved (Theorem \ref{2.5.1}). For $q=2$ we get an analogous diagram with $L^1$'s and $A_1$'s interchanged with $L^2$'s and $A_2$'s and the isomorphisms $\Psi_A$ and $(\Psi_0, \Psi_0)$ replaced by their inverses. 
Using self-adjointness of the Hecke operators in $\bfT_{n,\Oo}$ and arguing as in the proof of Lemma 2 of \cite{Taylor89}  one can show that $\alpha_{q,\bfC}^+\alpha_{q,\bfC} = \bmat N\fp +1 & T_{\fp}\\ T_{\fp} & (N\fp)^{k-2}(N \fp +1)\emat $, where $\alpha_{q,\bfC}$ (resp. $\alpha^+_{q,\bfC}$) denotes the complexification of $\alpha_q$ (resp. $\alpha^+_q$) and the map $\alpha_{q,\bfC}^+\alpha_{q,\bfC}$ makes sense as an automorphism of $L^1_{\bfC}(K_0)^2$. It follows from this and the fact that the composite in (\ref{diag90}) is an $\Oo$-module isomorphism that $\det (\alpha_{q,\bfC}^+ \alpha_{q,\bfC}) \in \Oo^{\times}$.
 Inverting the logic we conclude that there exist level-raising congruences (i.e., $C(\bfT_{\Oo})\neq 0$) if $\det (\alpha_{q,\bfC}^+ \alpha_{q,\bfC}) \not\in \Oo^{\times}$. We thus arrive at a conclusion that for a cuspidal Hecke eigenform $f$ on $\GL_2(\AF)$ of (parallel) weight $k$ and level $\fN$ such that  $a_{\fp}^2 \equiv (N\fp)^{k-2}(N\fp+1)^2 \pmod{\varpi}$ (where we write $a_{\fp}$ for the $f$-eigenvalue of $T_{\fp}$) there exists a cuspidal Hecke eigenform $g$ on $\GL_2(\AF)$ of weight $k$, level $\fN\fp$ (new at $\fp$) 
  with $f \equiv g$ (mod $\varpi$) - the congruence being between the Hecke eigenvalues for $T_{\fa}$ for $\fa$ a prime with $\fa \nmid \fN\fp$. This would be an analogue for $F$ of classical level-raising congruences of Ribet and Diamond \cite{Ribet84}, \cite{Diamond91}. 

However, the problem is that $A_2$ is not necessarily saturated due to the presence of torsion in the degree two cohomology (see e.g.  \cite{Sengun11} for some examples where this group contains torsion as well as the tables in section 9.8 of \cite{CalegariVenkateshbook} which contain results due to Dunfield exhibiting torsion in the first homology with $\bfZ$-coefficients which is isomorphic to the second degree cohomology with $\bfZ$-coeffcients). Let us discuss it in some more details and offer a weaker result (Theorem \ref{Ihara}) on the level of cohomology. 
%\com{Added:}
%See also Theorem \ref{CalVen} (due to Calgeari and Venkatesh) for a level-raising result for torsion homology classes.

By Lemma \ref{4.1.1}  there exists an integer $s$ such that $A_q^{\rm sat}/A_q \cong \ker (\alpha_{q,s})$. Consider the following commutative diagram:
% (the top arrow is injective by Lemma \ref{inj}):
 \be \label{diagIhar} \xymatrix{L^q_{\Oo}(K_0)^2 \ar@{^{(}->}[r]^{\alpha_q}\ar@{->>}[d] & L^q_{\Oo}(K_1)\ar@{->>}[d] \\
L^q_{\Oo}(K_0)^2\otimes\Oo/\varpi^s \ar[r]^{\alpha_{q,s}:=\alpha_q\otimes 1} \ar@{=}[d] & L^q_{\Oo}(K_1)\otimes\Oo/\varpi^s \ar@{=}[d]\\
\bigoplus_j (H^q_P(\Gamma_{0,j}, L(k-2,\Oo))^{\rm tf})^2\otimes \Oo/\varpi^s \ar[r]^{\alpha_q\otimes1} & \bigoplus_j (H^q_P(\Gamma_{1,j}, L(k-2,\Oo))^{\rm tf})\otimes \Oo/\varpi^s.}\ee 

\begin{thm} \label{Ihara} For every positive integer $s$ the map $$\alpha_q\otimes 1: \bigoplus_j (H^q_P(\Gamma_{0,j}, L(k-2,\Oo)))^2\otimes \Oo/\varpi^s \to \bigoplus_j (H^q_P(\Gamma_{1,j}, L(k-2,\Oo)))\otimes \Oo/\varpi^s$$ is injective. \end{thm}

It follows from Theorem \ref{Ihara} that the map $$\alpha_q:\bigoplus_j (H^q_P(\Gamma_{0,j}, L(k-2,\Oo)))^2 \to \bigoplus_j (H^q_P(\Gamma_{1,j}, L(k-2,\Oo)))$$ is injective and that the image of the first module is saturated in the second one. However, Theorem \ref{Ihara} implies injectivity of the bottom map in diagram (\ref{diagIhar}) only in the case of 
 $H^1$ which is torsion-free by Lemma \ref{tfH1}. Thus Theorem \ref{Ihara} falls short of proving that $A_2$ is saturated and we only obtain the following corollary:

\begin{cor} \label{sat} One has $A_1^{\rm sat} = A_1$. Moreover if $\bigoplus_jH^2_P(\Gamma_{i,j}, L(k-2, \Oo))$ is torsion-free for $i=0,1$, then $A^{\rm sat}_2=A_2$. \end{cor}

Finally, let us note that one cannot use Theorem \ref{Ihara} to prove the existence of level-raising congruences on the level of cohomology groups instead of the lattices $L$ because the pairings $J_n$ are only perfect modulo torsion (Theorem \ref{2.5.1}). 

\begin{proof} [Proof of Theorem \ref{Ihara}]
%From now on if we drop $q$ from notation it means that the statement refers to both $q=1$ and $q=2$.
We need to prove that the map $\alpha_q \otimes 1$ in Theorem \ref{Ihara} is injective. To shorten formulas for a $\bfZ$-module $A$ we put $\tilde{A}:= L(k-2, A)$. The short exact sequence $0 \to \tilde{\Oo} \xrightarrow{\cdot \varpi^s} \tilde{\Oo} \to \widetilde{\Oo/\varpi^s \Oo} \to 0$ gives rise to the following commutative diagram where the bottom two rows are exact sequences  and the objects in the top row are by definition the kernels of the vertical maps. \be \label{parcoh}\xymatrix{ \bigoplus_j H^q_P(\Gamma_{0,j}, \tilde{\Oo})\ar[r]^{\cdot \varpi^s}\ar[d]&\bigoplus_jH^q_P(\Gamma_{0,j}, \tilde{\Oo})\ar[r]\ar[d] & \bigoplus_j H^q_P(\Gamma_{0,j}, \widetilde{\Oo/\varpi^s})\ar[d]\\
\bigoplus_j H^q(\Gamma_{0,j}, \tilde{\Oo}) \ar[r]^{\cdot\varpi^s} \ar[d]&\bigoplus_j H^q(\Gamma_{0,j}, \tilde{\Oo}) \ar[r]\ar[d]&  \bigoplus_j H^q(\Gamma_{0,j}, \widetilde{\Oo/\varpi^s})\ar[d]\\
 \bigoplus_j \bigoplus_{B\in \mB_j} H^q(\Gamma_{0,j,B}, \tilde{\Oo})\ar[r]^{\cdot\varpi^s}&\bigoplus_j \bigoplus_{B\in \mB_j} H^q(\Gamma_{0,j,B}, \tilde{\Oo})\ar[r] &\bigoplus_j \bigoplus_{B\in \mB_j} H^q(\Gamma_{0,j,B}, \widetilde{\Oo/\varpi^s})  .}\ee
 One gets an identical diagram for the groups $\Gamma_{1,j}$.
By performing a subset of the proof of snake lemma, one can see that the top row is also exact provided that the bottom-left horizontal arrow is injective. For $q=1$ this follows from the exactness of the sequence of $H^0$'s. For $q=2$ this is a consequence of the following lemma: 

\begin{lemma} \label{ass1} The groups $H^2(\Gamma_{0,j,B},\tilde{\Oo})$ and $H^2(\Gamma_{1,j,B},\tilde{\Oo})$ are torsion-free. \end{lemma}
%\com{Note that this is satisfied for $k=2$ since the functor $\Hom(\Gamma_{i,j,B}, \cdot)$ is right exact as $\Gamma_{i,j,B}$ is free over $\bfZ$, hence projective.}
%On the other hand for $q=2$ this follows from torsion-freeness of $H^1(\Gamma_{0,j,B}, \tilde{\Oo})$ as we now explain. Indeed, one can see easily that taking the $\Gamma_{0,j,B)$-invariants in the short exact sequence $0 \to \tilde{\Oo} \to \tilde{E} \tilde{E/\Oo} \to 0$ is right-exact. Hence $H^1(\Gamma_{0,j,B}, \tilde{O})$ injects into $H^1(\Gamma_{0,j,B},E)$, thus in particular is torsion-free. 
% In fact we will show that the groups $\Gamma_{0,j,B}$ have cohomological dimension 1.
 %or $B \in \mB_j$ write $B=MN$ for its Levi decomposition. Note that $M \cap \Gamma_{0,j}$ is a finite group of order dividing $\#\OF^{\times}$, hence for $A=\tilde{\Oo}$ or $A=\tilde{\Oo/\varpi^s}$ we have $H^2(\Gamma_{0,j,B},A) = H^2(\Gamma_{0,j} \cap N, A)$. Since $N \cong F$, we see that $ \Gamma_{0,j} \cap N \cong \OF \cong \bfZ^2$ as abelian group. By the inflation-restriction sequence we have $0 \to H^2(\bfZ, A^{\bfZ}) \to H^2(\bfZ^2, A) 

\begin{proof} Let $B \subset \SL_2(F)$ be a Borel subgroup. Write $B=MN$ for its Levi decomposition. Note that $M \cap \Gamma_{0,j}$ is a finite group of order dividing $\#\OF^{\times}$, hence  we have $H^2(\Gamma_{0,j,B},\tilde{\Oo}) \hookrightarrow H^2(\Gamma_{0,j} \cap N, \tilde{\Oo})$, so it is enough to show that the latter group is torsion-free.

 It is easy to see that $ \Gamma_{0,j} \cap N  \cong \bfZ^2$ as abelian groups. Write $G$ for $\Gamma_{0,j} \cap N \cong \bfZ^2$ and choose a subgroup $H \subset G$ such that $H\cong \bfZ$. From the Hochshild-Serre spectral sequence one concludes that the sequence
\be \label{HS1}  H^2(G/H, \tilde{\Oo}^H) \to H^2(G,\tilde{\Oo})^* \to H^1(G/H, H^1(H,\tilde{\Oo})),\ee where $H^2(G,M)^* = \ker({\rm res}(H^2(G,M) \to H^2(H,M))$ is exact. 
%(cf. e.g. \cite{MilneCFT}, Remark 1.35).
 Since $H$ has cohomological dimension one we obtain that $H^2(G, \tilde{\Oo})^* = H^2(G, \tilde{\Oo})$ and that $H^2(G/H, \tilde{\Oo}^H)=0$. Hence we conclude from (\ref{HS1}) that $H^2(G,\tilde{\Oo})$ injects into $H^1(G/H, H^1(H,\tilde{\Oo}))$ and thus it is enough to prove that the latter group is torsion-free. 

Note that since the sequence of $H$-invariants $$0 \to \tilde{\Oo}^H \to (\tilde{\Oo}\otimes E)^H \to (\tilde{\Oo}\otimes E/\Oo)^H \to 0$$ is exact 
and $H$ has cohomological dimension one, we obtain a short exact sequence of cohomology groups \be \label{HS2} 0 \to H^1(H, \tilde{\Oo}) \to  H^1(H, \tilde{\Oo}\otimes E) \to H^1(H, \tilde{\Oo}\otimes E/\Oo) \to 0.\ee Since all  the groups in (\ref{HS2}) are naturally $G/H$-modules we can compute their long exact sequence of cohomology for the group $G/H$ and obtain:
\begin{multline}  \label{HS3} 0 \to  H^1(H, \tilde{\Oo})^{G/H} \to  H^1(H, \tilde{\Oo}\otimes E)^{G/H} \xrightarrow{\phi} H^1(H, \tilde{\Oo}\otimes E/\Oo)^{G/H} \to \\
\to H^1(G/H,H^1(H, \tilde{\Oo})) \to  H^1(G/H, H^1(H, \tilde{\Oo}\otimes E)) \end{multline}
Since the last group in (\ref{HS3}) is an $E$-vector space it is clearly torsion-free, so it suffices to prove that $\phi$ is surjective. 

To do this, first note that since $G$ is abelian the group $G/H$ acts on $f \in H^1(H,M)$ by $(\gamma\cdot f)(h) = \gamma \cdot f(h)$. Moreover, since $H \cong \bfZ$ one has $H^1(H, M) \cong \frac{M}{(\sigma-1)M}$, where $\sigma$ is a generator of $H$ and 1 denotes the identity in $H$ (hence zero in $\bfZ$). So, our claim follows from exactness of the following sequence (which is easy to show):

\be \label{HS4} 0 \to \left(\frac{\tilde{\Oo}}{\bmat 0 & 1 \\ 0 & 0 \emat \tilde{\Oo}}\right)^{G/H} \to  \left(\frac{\tilde{\Oo}\otimes E}{\bmat 0 & 1 \\ 0 & 0 \emat \tilde{\Oo}\otimes E}\right)^{G/H} \to \left(\frac{\tilde{\Oo}\otimes E/\Oo}{\bmat 0 & 1 \\ 0 & 0 \emat \tilde{\Oo}\otimes E/\Oo}\right)^{G/H}\to 0 \ee
The same proof works for $\Gamma_{1,j,B}$. 
\end{proof}

 Using the exactness of the top row in (\ref{parcoh}) we get the following commutative diagram 
$$\xymatrix{ (\bigoplus_j H^q_P(\Gamma_{0,j}, \tilde{\Oo}))^2\otimes \Oo/\varpi^s \ar@{=}[r] \ar[d]^{\alpha_q \otimes 1} &  \frac{(\bigoplus_j H_P^q(\Gamma_{0,j}, \tilde{\Oo}))^2}{\varpi^s (\bigoplus_j H_P^q(\Gamma_{0,j}, \tilde{\Oo}))^2} \ar@{^{(}->}[r] \ar[d] & (\bigoplus_j H^q_P(\Gamma_{0,j}, \widetilde{\Oo/\varpi^s}))^2 \ar[d]^{\alpha_{q,\Oo/\varpi^s}}\\ (\bigoplus_j H^q_P(\Gamma_{1,j}, \tilde{\Oo}))\otimes \Oo/\varpi^s \ar@{=}[r]  &  \frac{(\bigoplus_j H^q_P(\Gamma_{1,j}, \tilde{\Oo}))}{\varpi^s (\bigoplus_j H^q_P(\Gamma_{1,j}, \tilde{\Oo}))} \ar@{^{(}->}[r]  & (\bigoplus_j H^q_P(\Gamma_{1,j}, \widetilde{\Oo/\varpi^s})),}$$ where $\alpha_{q,\Oo/\varpi^s}$ is induced from $\alpha_q$.
 The  map $\alpha_{q,\Oo/\varpi^s}$ is injective, by Theorem \ref{IharaH1} when $q=1$ and by Theorem \ref{IharaH2} when $q=2$. Hence the first vertical arrow is injective. \end{proof}

%\com{Added:}
%Let us close this section with a level-raising result for imaginary quadratic fields which is due to Calegari and Venkatesh. Here $Y_0$ and $Y_1$ are as defined in section \ref{Preliminaries}. 
%\begin{thm} [Calegari-Venkatesh \cite{CalegariVenkateshbook}, Theorem C] \label{CalVen} Let $\ell$ be a prime. Suppose $[c] \in H_1(Y_0, \bfF_{\ell})$ is a non-Eisenstein Hecke eigenclass, in the kernel of $T_{\fp}^2-(N\fp +1)^2$. Then there exists $[\tilde{c}] \in H_1(Y_1, \bfF_{\ell})$ with the same generalized Hecke eigenvalues as $c$ at primes not dividing $\fp$, and not in the image of the pullback degeneracy map $H_1(Y_0, \bfF_{\ell})^2 \to H_1(Y_1, \bfF_{\ell})$. \end{thm}
%For a precise definition of \emph{non-Eisenstein} we refer the reader to \cite{CalegariVenkateshbook}, section 3.8. See also [loc.cit.] Theorem 4.3.1 and Remark 4.3.2 which explains the relation of the above result to the classical level-raising problem. In particular as it is noted in Remark 4.3.2 the resulting homology classes may not lift to characteristic zero. 

\section{Acknowledgements}
We would like to thank Tobias Berger and Eknath Ghate for many helpful conversations and answering many of our questions. 
We would also like to express our gratitude to the Max-Planck-Institut in Bonn where part of this work was carried out for its hospitality during the author's stay there in the Summer of 2011. Finally, we would like thank Frank Calegari and Akshay Venkatesh for sending us an early version of their book.

\bibliographystyle{amsplain}
\bibliography{standard2}

\end{document}